\documentclass{amsart}
\usepackage{amsmath,amssymb, amscd}\newtheorem{dfn}{Definition}[section]
\newtheorem{thm}[dfn]{Theorem}
\newtheorem{lem}[dfn]{Lemma}
\newtheorem{cor}[dfn]{Corollary}
\newtheorem{rem}[dfn]{Remark}
\newtheorem{ex}[dfn]{Example}
\newtheorem{prop}[dfn]{Proposition}
\newtheorem{conj}[dfn]{Conjecture}

\newcommand{\E}{\mathcal{E}}
\newcommand{\D}{\mathcal{D}}
\newcommand{\Db}{D^b}

\newcommand{\OO}{\mathcal{O}}
\newcommand{\RHom}{\mathbf{R}\mathrm{Hom}}
\newcommand{\Hom}{\mathrm{Hom}}
\newcommand{\hcat}{h_{\mathrm{cat}}}
\newcommand{\htop}{h_{\mathrm{top}}}
\newcommand{\CC}{\mathbb{R}}
\newcommand{\RR}{\mathbb{R}}
\newcommand{\ZZ}{\mathbb{Z}}
\newcommand{\ext}{\mathrm{ext}}
\newcommand{\Ext}{\mathrm{Ext}}
\newcommand{\Stab}{\mathrm{Stab}}

\newcommand{\Aut}{\mathrm{Aut}}

\newcommand{\Lf}{\mathbf{L}f^*}
\begin{document}
\title[ON ENTROPY OF SPHERICAL TWISTS]{ON ENTROPY OF SPHERICAL TWISTS}
\author[Genki Ouchi]{GENKI OUCHI\\
with an appendix by Arend Bayer}
\email{genki.oouchi@ipmu.jp}
\address{Graduate School of Mathematical Sciences, University of Tokyo, Meguro-ku, Tokyo 153-8914, Japan}
\address{School of Mathematics and Maxwell Institute,
University of Edinburgh,
James Clerk Maxwell Building,
Peter Guthrie Tait Road, Edinburgh, EH9 3FD,
United Kingdom}
\email{arend.bayer@ed.ac.uk}\maketitle
\begin{abstract}
In this paper, we compute categorical entropy of spherical twists. In particular, we prove that Gromov-Yomdin type conjecture holds for spherical twists. Moreover, we construct counterexamples of Gromov-Yomdin type conjecture for K3 surfaces modifying Fan's construction for even higher dimensional Calabi-Yau manifolds. 

The appendix, by Arend Bayer, shows non-emptiness of complements of a number of spherical objects in
the derived categories of K3 surfaces.
\end{abstract}
\section{Intorduction}
\subsection{Motivation and Results}

 A pair  $(X, f)$ of a topological space $X$ and a continuous self map $f : X \to X$ is called a topological dynamical system. The topological entropy $\htop(f)$ of $f$ is the fundamental invariant of a topological dynamical system $(X,f)$, that measures the complexity of topological dynamical systems. Recently, Dimitrov, Haiden, Katzarkov and Kontsevich \cite{DHKK} studied categorical dynamical systems. A categorical dynamical system $(\D, \Phi)$ is a pair of a triangulated category $\D$ and an exact endofunctor $\Phi : \D \to \D$. They introduced the categorical entropy $h_t(\Phi)$ of a categorical dynamical system $(\D,\Phi)$ as a real-valued function on real numbers. We denote the special value $h_0$ by $\hcat$. Regarding endomorphisms of algebraic varieties as topological dynamical systems, there are cross-disciplinary studies in the algebraic geometry and the complex dynamics \cite{Ogu}. In \cite{Ouc}, the author constructed  examples of automorphisms of positive topological entropy on hyperK\"ahler manifolds via autoequivalences of positive categorical entropy on derived categories of K3 surfaces. Then the hyperK\"ahler manifolds are constructed as moduli spaces of stable objects on K3 surfaces \cite{BM:walls}. Moreover, we compared their categorical and topological entropy. The case of abelian surfaces is studied in \cite{Yos}. In the setting of algebraic geometry, we can compute topological entropy by Gromov-Yomdin theorem.
 
 \begin{thm}\rm{}(\cite{Gro1}, \cite{Gro2}, \cite{Yom})\label{GYthm}\it{}
 Let $X$ be a smooth projective variety over $\mathbb{C}$.  For a surjective endomorphism $f: X \to X$, we have
 \[\htop(f)=\log\rho(f^*|_{\oplus_{p=0}^{\dim X}H^{p,p}(X,\ZZ)}). \]
 Here, $\rho$ is the spectral radii of linear maps and we consider the analytic topology on $X$.
 \end{thm}
 In the same setting, Kikuta and Takahashi proved the following theorem.
 
 \begin{thm}\rm{}(\cite{KT})\label{cattop}\it{}
 Let $X$ be a smooth projective variety over $\mathbb{C}$.  For a surjective endomorphism $f: X \to X$, we have
 \[\hcat(\Lf)=\log\rho(f^*|_{\oplus_{p=0}^{\dim X}H^{p,p}(X,\ZZ)}).\]
 In particular, $\hcat(\Lf)=\htop(f)$ holds.
  \end{thm}
  By Theorem \ref{cattop}, it looks like natural to expect the Gromov-Yomdin type conjecture.
  
  \begin{conj}\rm{}(\cite{KT})\it{}\label{KTconj}
  Let $X$ be a smooth projective variety over $\mathbb{C}$. For an autoequivalence $\Phi \in \Aut(D^b(X))$, we have 
  \[\hcat(\Phi)=\log\rho([\Phi]).\]
  Here, $[\Phi] : K_{\mathrm{num}}(X) \to K_{\mathrm{num}}(X)$ is the induced linear map on the numerical Grothendieck group  $K_{\mathrm{num}}(X)$ of 
  $X$.
  \end{conj}
  
  Conjecture \ref{KTconj} is known to be true for smooth projective curves \cite{K}, smooth projective varieties with the ample (anti)canonical line bundles \cite{KT}, abelian surfaces \cite{Yos}. Moreover, it is also true for shifts functors \cite{DHKK}, automorphisms of varieties \cite{KT} and line bundle tensors \cite{DHKK} on any smooth projective varieties.  Kikuta, Shiraishi and Takahashi \cite{KST} proved the lower bounds in Conjecture \ref{KTconj} on perfect derived categories of smooth proper dg algebras. In this paper, we compute categorical entropy of spherical twists.
  
  \begin{thm}\label{mainintr}
Let $\D$ be the perfect derived category of a smooth proper dg algebra. Let $T_\E$ be the spherical twist of a $d$-spherical object $\E \in \D$. 
For $t  \leq 0$, we have $h_t(T_\E)=(1-d)t$. In particular, $\hcat(T_\E)=0$ holds.
Assume that $^{\perp}\E \neq \emptyset$. Then we have $h_t(T_\E)=0$ for $t>0$.
\end{thm}
In general, it is not known whether orthogonal complements of spherical objects are non-empty. In the
appendix by Arend Bayer, we show that $^{\perp}\mathcal{E}$ is not empty when $\E$ is a
Gieseker-stable spherical sheaf on a K3 surface, or any spherical object
in the derived category of a K3 surface of Picard number one.
\begin{cor}
Let $X$ be a K3 surface of Picard number one. For a spherical object $\E \in D^b(X)$ and $t \in \mathbb{R}$, 
we have 
\[ h_t(T_\E):=\begin{cases}
    0 & (t \geq 0) \\
  -t  & (t<0).
  \end{cases}
\]
\end{cor}

On  the derived categories of Calabi-Yau algebras associated with acyclic quivers, Ikeda \cite{Ike} computed categorical entropy of spherical twists without the assumption in Theorem \ref{mainintr}. 

Recently, Fan \cite{Fan} found counterexamples of Conjecture \ref{KTconj} on $2d$-dimensional Calabi-Yau hypersurfaces  for $d\geq2$. Modifying his proof, we construct counterexamples of Conjecture $\ref{KTconj}$ for any K3 surfaces.
  
 \begin{prop}
 Let $X$ be a K3 surface and $H$ be a very ample line bundle on $X$ with $H^2=2d$. Define the autoequivalence
 \[ \Phi:=T_{\OO_X} \circ (- \otimes \OO_{X}(-H)) \]
 of $D^b(X)$. 
We have 
\[ h_0(\Phi) \geq \log \rho(d+2) \]
and 
\[ \rho(\Phi^H) = \begin{cases}
    1 & (d =1,2,3,4) \\
   \frac{d-2+\sqrt{d^2-4d}}{2}  & (d \geq 5).
  \end{cases}\]
  In particular, we obtain the inequality $h_0(\Phi) > \log \rho(\Phi^H)$.
  \end{prop}
  
 It is interesting to study which autoequivalences satisfy Conjecture \ref{KTconj}.  In the last section, we remark that Conjecture \ref{KTconj} is true for automorphisms of smooth projective varieties over any algebraically closed fields. In the original proof of Theorem \ref{cattop}, they used Kodaira vanishing theorem. However,  Fujita vanishing theorem \cite{Fuj} is enough in this situation.
So we can generalize Theorem \ref{cattop} for any algebraically closed fields. 

\subsection{Notation and Convention}
All triangulated categories are linear over fields. For a smooth projective variety $X$, we use the following notation.
For coherent sheaves $E$ and $F$ on $X$, we set 
\[h^i(E):=\dim H^i(E), \ext^i(E,F):= \dim \Ext^i(E,F) \]
for $i \in \ZZ$. 
For a smooth proper dg category $\D$ over a field $k$ and $d \in \ZZ_{>0}$, a $d$-spherical object $\E \in \D$ is a spherical object satisfying $\RHom(\E,\E)=k \oplus k[-d]$.
For a liner map $f : V \to V$ on a finite dimensional vector space over $\mathbb{Q}$ or $\mathbb{R}$ or $\mathbb{C}$, the spectral radius $\rho(f)$  of  $f$ is the maximum of  absolute values of complex eigenvalues of $f$.

\subsection*{Acknowledgements}
The author would like to express his sincere gratitude to Professor Yukinobu Toda for valuable comments and warmful encouragement.  He would like to thank Yu-Wei Fan for showing the draft of the paper \cite{Fan} to him. Advice and comments given by Wahei Hara and Professor Kota Yoshioka  has been a great help in the formulation of Proposition 4.3. In the first version of this paper, he treated only quartic K3 surfaces.  He would also like to thank Akishi Ikeda, Kotaro Kawatani, Kohei Kikuta, Naoki Koseki for valuable conversation and comments.  This work is supported by Grant-in-Aid for JSPS Research Fellow 15J08505.

\section{Preliminaries} \label{sect:prelims}
\subsection{Categorical entropy}
In this section, we recall the fundamental results for categorical entropy.
Let $\D$ be a  triangulated category. For an object $G \in \D$, we denote  $\langle G\rangle_{\mathrm{thick}}$ as the smallest full triangulated subcategory of $\D$, that contains $G$ and closed under taking direct summand. An object $G$ is a split generator if $\D=\langle G \rangle_{\mathrm{thick}}$ holds. Assume that $G$ is a split generator of $\D$.
\begin{dfn}$($\cite{DHKK}$)$
Let $E$ be an object in $\mathcal{D}$. For a real number $t \in \mathbb{R}$,
we define the complexity $\delta_{t}(G,E)$ of $F$ with respect to $E$ as follow:
\[ \delta_{t}(G,E):=\inf \biggl\{\sum_{i=1}^{k}e^{n_i t} \mid E_{i-1} \to E_i \to G[n_i]  (1 \leq i \leq k) ,E_0=0, E_k=E \oplus E^{\prime}\biggr\}. \] 
\end{dfn}
The definition of categorical entropy is as follow.

\begin{thm}$($\cite{DHKK}$)$
Let $G, G' \in \D$ be split generators. For an exact $\Phi : \D \to \D$ and a real number $t \in \mathbb{R}$, we define the entropy of $\Phi$ as
\[ h_t(\Phi):=\lim_{n \to \infty}\log\delta_t(G, \Phi^n(G'))/n. \]
This limit exists in the interval $[-\infty, \infty)$ and is independent to a choice of $G$ and $G'$. Moreover, we define $\hcat(\Phi):=h_0(\Phi) \in \mathbb{R}_{\geq0}$.
\end{thm}

There is another characterization of categorical entropy.

\begin{dfn}$($\cite{DHKK}$)$
Assume that $\D$ is Ext-finite. For objects $E,F \in \D$ and a real number $t \in \mathbb{R}$, we set
\[ \delta'(E, F):=\sum_{m \in \ZZ}\ext^m(E,F)e^{-mt}. \]
For simplicity, we put $\delta'(E,F):=\delta'_0(E,F)$.
\end{dfn}
We will use Theorem \ref{delta'} in the proof of main results.
\begin{thm}\label{delta'} $($\cite{DHKK}$)$
Assume that $\D$ is a smooth proper dg category. Let $G$ and $G'$ be split generators of $\D$. For an exact endofunctor $\Phi : \D \to \D$, we have
\[h_t(\Phi)=\lim_{n \to \infty}\log\delta^\prime(G, \Phi^n(G'))/n.\]
\end{thm}
At the end of this section, we recall the following facts.

\begin{rem}
\begin{itemize}
\item[(1)]Let $\D$ be the perfect derived category of a smooth proper dg algebra $A$.  Then $A \in \D$ is a split generator. \rm{}(\cite{KS})

\item[(2)] \it{}Let $X$ be a smooth projective variety over a field $k$. Let $\OO_X(1)$ be a very ample line bundle on $X$. For $k \in \ZZ$, an object 
$G:= \oplus_{i=0}^{\dim X+1}\OO_X(i+k)$ is a split generator of $D^b(X)$. \rm{}(\cite{Or0})
\end{itemize}
\end{rem}

\subsection{Mukai lattices of K3 surfaces}
In this subsection, we recall the relation between numerical Grothendieck groups and algebraic Mukai lattices of K3 surfaces. 
Let $X$ be a K3 surface. 
\begin{dfn}
We define the algebraic Mukai lattice  $\widetilde{H}^{1,1}(X,\mathbb{Z})$ of $X$ as follow.
\[ \widetilde{H}^{1,1}(X,\mathbb{Z}):= H^0(X,\ZZ) \oplus \mathrm{NS}(X) \oplus H^4(X,\ZZ) \]
The lattice structure is defined by 
\[  \langle(r_1,c_1,m_1),(r_2,c_2,m_2) \rangle := c_1c_2-r_1m_2-r_2m_1  \]
for $(r_1,c_1,m_1),(r_2,c_2,m_2)  \in  \widetilde{H}^{1,1}(X,\mathbb{Z})$.
\end{dfn}

\begin{rem}\rm{(\cite{Huy}, Section 5.2.)}\it{}
The homomorphism 
\[ v : K_{\mathrm{num}}(X) \to \widetilde{H}^{1,1}(X,\mathbb{Z}), [E] \mapsto \mathrm{ch}(E)\sqrt{\mathrm{td}(X)}\]
induces the isometry between $(K_{\mathrm{num}}(X), -\chi)$ and $(\widetilde{H}^{1,1}(X,\mathbb{Z}), \langle-, -\rangle)$ by Riemann-Roch formula. 
For an autoequivalence $\Phi \in \Aut(D^b(X))$, there is the following diagram.
$$\begin{CD}
K_{\mathrm{num}}(X) @> [\Phi] >>K_{\mathrm{num}}(X)\\
@V v VV @VV v V\\
\widetilde{H}^{1,1}(X,\mathbb{Z})  @> \Phi^H  >> \widetilde{H}^{1,1}(X,\mathbb{Z})
\end{CD}$$
In particular, we have $\rho(\Phi^H)=\rho([\Phi])$.
\end{rem}
 \section{Computation of entropy of spherical twists}
Let $\D$ be a perfect derived category of a smooth proper dg algebra. Take a split generator $G \in \D$. For a $d$-spherical object $\E \in \D$, the spherical twist functor $T_\E$ fits into the following exact triangle.
\begin{equation*}
\RHom(\E,-) \otimes \E \to \mathrm{id} \to T_\E 
\end{equation*}
\begin{thm}\label{main}
Let $\E \in \D$ be a $d$-spherical object. 
For $t  \leq 0$, we have $h_t(T_\E)=(1-d)t$. In particular, $\hcat(T_\E)=0$ holds.
Assume that $d=1$ or $^{\perp}\E \neq \emptyset$. Then we have $h_t(T_\E)=0$ for $t>0$.
\end{thm}
\begin{proof}
Let $G, G^\prime \in \D$ be split generators. Using the exact triangle with respect to the spherical twist, we have 
\[ T_\E^{n-1}(G) \to T_\E^n(G) \to \RHom(\E,G) \otimes \E[(n-1)(1-d)+1]. \] 
Here, note that 
\begin{eqnarray*}
\RHom(\E,T_\E^{n-1}(G))&=&\RHom(T_\E^{-(n-1)}(\E),G)\\
&=&\RHom(\E,G[(n-1)(1-d)])\\
&=&\RHom(\E,G)[(n-1)(1-d)].
\end{eqnarray*}
Fix $t \in \RR$.
Then we have 
\begin{align*}
&\delta^\prime_t(G^\prime, T_\E^n(G))\\
&\leq \delta^\prime_t(G^\prime, T_\E^{n-1}(G))+\delta^\prime_t(G^\prime, \RHom(\E,G) \otimes \E[(n-1)(1-d)+1])\\
&=\delta^\prime_t(G^\prime, T_\E^{n-1}(G))+ \delta^\prime_t(\E,G)\delta^\prime_t(G^\prime,\E)e^{(n(1-d)+d)t}\\
&\leq \delta^\prime_t(G^\prime,G)+ \delta^\prime_t(\E,G)\delta^\prime_t(G^\prime,\E)\sum_{k=1}^{n}e^{(k(1-d)+d)t}\\
&\leq \delta^\prime_t(G^\prime,G)+ n\delta^\prime_t(\E,G)\delta^\prime_t(G^\prime,\E)f_n(t).
\end{align*}

Here, we put $f_n(t):=\mathrm{max}\{1, e^{(n(1-d)+d)t}\}$. 
We assume that $d\geq2$ and $t \leq 0$. Due to $f_n(t)=e^{(n(1-d)+d)t}$,
we have 
\[ h_t(T_\E) \leq (1-d)t.\]
Since $G \oplus \E$ is a split generator of $\D$,
\begin{eqnarray*}
h_t(T_\E)&=& \lim_{n \to \infty}\log\delta^\prime_t(G^\prime, T_\E^n(G \oplus \E ))/n\\
&\geq&\lim_{n \to \infty}\log\delta^\prime_t(G^\prime, T_\E^n(\E ))/n\\
&=&\lim_{n \to \infty}\log\delta^\prime_t(G^\prime, \E[n(1-d)])/n\\
&=&\lim_{n \to \infty}\log\delta^\prime_t(G^\prime, \E )e^{n(1-d)t}/n\\
&=&(1-d)t.
\end{eqnarray*}
Hence, we obtain $h_t(T_\E)=(1-d)t$.
If $d=1$, then we have $h_t(T_\E) \leq 0$. Note that $T_\E(\E)=\E$ and put $A:=\E$. 
If $d\geq2, t>0$, then we have $f_n(t)=1$. So we obtain 
\[h_t(T_\E) \leq 0.\]
If we assume that $^\perp\E \neq \emptyset$,  we can take $A \in ^\perp\E$ and note that $T_\E(A)=A$.
Since $G \oplus A$ is a split generator of $\D$,
\begin{eqnarray*}
h_t(T_\E)&=& \lim_{n \to \infty}\log\delta^\prime_t(G^\prime, T_\E^n(G \oplus A))/n\\
&\geq&\lim_{n \to \infty}\log\delta^\prime_t(G^\prime, T_\E^n(A ))/n\\
&=&\lim_{n \to \infty}\log\delta^\prime_t(G^\prime, A)/n\\
&=&0
\end{eqnarray*}
Hence, we obtain $h_t(T_\E)=0$.
\end{proof}
We can immediately give examples of spherical objects satisfying the assumption in Theorem
\ref{main}; see Proposition A.1 for more cases.
\begin{ex}
Let $X$ be a K3 surface with a spherical vector bundle $\E$ of rank $r \geq 2$.  Then there is a vector bundle $A$ such that $\E\mathrm{nd}(\E,\E)=A \oplus \OO_X$. Then we have $\RHom(\OO_X, A)=0$ i.e. $A \in {^{\perp}\OO}_X$.
\end{ex}

\begin{ex}
Let $X$ be a smooth projective surface. For a $(-2)$ curve $C$ on $X$, the structure sheaf $\OO_C$ of $C$ is a spherical object in $D^b(X)$. For a point $x \in X \setminus C$,  we have $\OO_x \in  {^{\perp}\OO_C}$.
\end{ex}

\section{Counterexample of Gromov-Yomdin type conjecture for K3 surfaces}
In this section, we give  counterexamples of Gromov-Yomdin type conjecture (Conjecture \ref{KTconj}) for K3 surfaces.
Let $X$ be a K3 surface and $\OO_X(1)$ be a very ample line bundle on $X$. We set $H:=\mathrm{c_1}(\OO_X(1))$ snd $2d:=H^2$. We take two split generators \[G:= \bigoplus_{i=1}^{3}\OO_X(i), G':=\bigoplus_{i=1}^{3}\OO_X(-i)\]
of $D^b(X)$.
We consider the autoequivalence $\Phi:=T_{\OO_X} \circ (- \otimes \OO_X(-1))$ of $D^b(X)$.
\begin{lem}\label{ext}
For $i, k \in \ZZ_{>0}$, we have
\[ \Ext^m(\OO_X, \Phi^n(\OO_X(-i)) \otimes \OO_X(-k))=0\]
for $m \notin [2, n+2]$.
Moreover, we get the isomorphism
\[ \Ext^{n+2}(\OO_X, \Phi^n(\OO_X(-i)) \otimes \OO_X(-k)) \simeq \Ext^{n+1}(\OO_X, \Phi^{n-1}(\OO_X(-i)) \otimes \OO_X(-1))^{h^0(\OO_X(k))}.\]\end{lem}
\begin{proof}
We prove it by induction on $n$.
Consider the canonical exact triangle.
\begin{align*}
\RHom(\OO_X, \Phi^{n-1}(\OO_X(-i)) &\otimes \OO_X(-1)) \otimes \OO_X \\ &\to \Phi^{n-1}(\OO_X(-i)) \otimes \OO_X(-1)
 \to \Phi^n(\OO_X(-i)). 
  \end{align*}
  Applying $\otimes \OO_X(-k)$ and $\RHom(\OO_X, -)$, we get the exact triangle
  \begin{align*}
&\RHom(\OO_X, \Phi^{n-1}(\OO_X(-i)) \otimes \OO_X(-1)) \otimes \RHom(\OO_X, \OO_X(-k))\to \\ &\quad \RHom(\OO_X, \Phi^{n-1}(\OO_X(-i)) \otimes \OO_X(-k-1))
 \to \RHom(\OO_X, \Phi^n(\OO_X(-i))\otimes \OO_X(-k)). 
  \end{align*}
  Note that $\RHom(\OO_X, \OO_X(-k))=\Ext^2(\OO_X, \OO_X(-k))[-2]$ by Kodaira vanishing theorem.
  Due to induction hypothesis, we get the desired results by taking the long exact sequence and Serre duality.
\end{proof}

\begin{lem}\label{a}
For $i\in \ZZ_{>0}$, we have
\[ \delta'(\OO_X, \Phi^n(\OO_X(-i)) \otimes \OO_X(-1)) \geq (d+2)^{n}.\]
\end{lem}
\begin{proof}
By Lemma \ref{ext}, there is the isomorphism
  \[ \Ext^{n+2}(\OO_X , \Phi^n(\OO_X(-i))\otimes \OO_X(-1)) \simeq \Ext^{n+1}(\OO_X, \Phi^{n-1}(\OO_X(-i)) \otimes \OO_X(-1))^{h^0(\OO_X(1))}.\]
 
  Hence, we get 
  \begin{align*}
&\delta'(\OO_X, \Phi^n(\OO_X(-i)) \otimes \OO_X(-1))\\
&\geq\ext^{n+2}(\OO_X, \Phi^n(\OO_X(-i))\otimes \OO_X(-1)))\\
&=\ext^{n+1}(\OO_X, \Phi^{n-1}(\OO_X(-i))\otimes \OO_X(-1)))\cdot h^0(\OO_X(1))\\
&=\ext^2(\OO_X, \OO_X(-i-1)) \cdot h^0(\OO_X(1))^n \\
&\geq h^0(\OO_X(1))^n . 
  \end{align*} 
  Note that $\ext^2(\OO_X(-i-1))=h^0(\OO_X(i+1))\geq1$.
Finaly, we have 
   \begin{eqnarray*}
h^0(\OO_X(1))&=&\chi(\OO_X(1))\\
&=&d+2
\end{eqnarray*} 
by Kodaira vanishing and Riemann-Roch formula.   \end{proof}

\begin{prop}
We have 
\[ h_0(\Phi) \geq \log \rho(d+2) \]
and 
\[ \rho(\Phi^H) = \begin{cases}
    1 & (d =1,2,3,4) \\
   \frac{d-2+\sqrt{d^2-4d}}{2}  & (d \geq 5).
  \end{cases}\]
  In particular, we obtain the inequality $h_0(\Phi) > \log \rho(\Phi^H)$.
  \end{prop}
  \begin{proof}
  Using $\delta'(G, \Phi^n(G')) \geq \delta'(\OO_X, \Phi^n(\OO_X(-1)) \otimes \OO_X(-1))$ 
and Lemma \ref{a}, we have 
\begin{eqnarray*}
h_0(\Phi)
&=&\lim_{n \to \infty}\log\delta^\prime(G, \Phi^n(G'))/n\\
&\geq&\log h^0(\OO_X(1))\\
&=& \log(d+2). 
\end{eqnarray*}
We compute the spectral radius $\rho(\Phi^H)$. Consider the sublattice
\[L_d:=H^0(X,\ZZ) \oplus \ZZ \cdot H \oplus H^4(X,\ZZ) \]  
of $\widetilde{H}^{1,1}(X,\mathbb{Z})$.
The representation matrix of $\Phi^H|_{L_d}$ with respect to the standard basis $(1,0,0), (0,H,0), (0,0,1)$ is 
\[\begin{pmatrix}
                     -d &2d & -1 \\ -1 & 1 & 0\\ -1&0 & 0
                     \end{pmatrix}.\]

 Computing eigenvalues of it, we have 
 \[ \rho(\Phi^H|_{L_d}) = \begin{cases}
    1 & (d =1,2,3,4) \\
   \frac{d-2+\sqrt{d^2-4d}}{2}  & (d \geq 5).
  \end{cases}\]
Let $H^\perp$ be the orthogonal complement of $H$ in $\mathrm{NS}(X)$.  Since $\Phi^H(0,D,0)=(0,D,0)$ for $D \in H^\perp$,  we get $\rho(\Phi^H)=\rho(\Phi^H|_{L_d})$.
  \end{proof}

\section{Categorical entropy of surjective endomorphisms over algebraically closed field} 
In this section, we generalize Kikuta and Takahashi's result (Theorem \ref{cattop}) to any algebraically closed fields.
Let $X$ be a smooth projective variety over an algebraically closed field $k$. Let $f : X \to X$ be a surjective endomorphism of $X$. We recall Fujita vanishing theorem.

\begin{thm}\label{Fujita}$($\cite{Fuj}$)$
Let $H$ be an ample divisor on $X$. For a coherent sheaf $E$ on $X$, there is an integer $m(E, H)$ such that we have
\[H^i(E(mH+D))=0\]
for all $i>0, m \geq m(E,H)$ and any nef divisor $D$ on $X$.　
\end{thm}
\begin{thm}
Consider the derived pullback
\[ \Lf : D^b(X) \to D^b(X) \]
of $f$.
Then we have
\[ \hcat(\Lf)=\log \rho([\Lf]).\]
\end{thm}
\begin{proof}
By the result of \cite{KST}, we have $\hcat(\Lf) \geq \rho([\Lf])$. 
Note that we can take a very ample divisor $H$ on $X$ such that $m(\OO_X, H)=1$ . Then we consider the split generators
$ G:=\bigoplus_{i=1}^{\dim X+1} \OO_X(iH), G':=\bigoplus_{i=1}^{\dim X+1} \OO_X(-iH)$
of $D^b(X)$.
By Theorem \ref{Fujita}, we have
\begin{eqnarray*}
 \RHom(G', f^{*n}G)&=&\mathrm{Hom}(\OO_X, f^{*n}G \otimes G)\\
                                  &=&\mathrm{Hom}(G', f^{*n}G).
\end{eqnarray*}
Since $\delta'(G', f^{*n}G)=\chi(G', f^{*n}G)$,
we obtain
\begin{eqnarray*}
\hcat(\Lf)&=&\lim_{n \to \infty}\log\delta^\prime(G', f^{*n}G)/n\\
              &=&\lim_{n \to \infty}\log\chi(G', f^{*n}G)/n\\
              &\leq& \log \rho([\Lf]).
\end{eqnarray*}
\end{proof}

\appendix
\section{Complement of spherical objects on K3 surfaces \\
by Arend Bayer}

In this appendix, we prove that the complement of a spherical object on a K3 surface $X$ is
non-empty whenever that spherical object can be made stable for \emph{some} stability condition
on $\Db(X)$, in particular for all spherical objects on K3 surfaces of Picard rank one.
This partially answers a question by Daniel Huybrechts \cite[Chapter
16]{Daniel:K3-book}, and shows that Theorem 3.1 applies in a large number of cases. Our proof is
based on Bridgeland stability conditions as constructed in \cite{Bridgeland:K3} and wall-crossing
 in particular the analysis of divisorial contractions of \emph{Brill-Noether divisors} in
\cite{BM:walls}.

Throughout this appendix, we let $X$ be a smooth complex K3 surface. We denote by $\Stab^\dagger(X)$
the connected component of the space of stability conditions constructed in \cite{Bridgeland:K3}.

\begin{prop} \label{prop:nonemptyspherical}
Let $X$ be a K3 surface, and let $\mathcal{E}$ be a spherical object that is $\sigma$-stable for some
stability condition $\sigma = (Z, \mathcal{P})$ in the distinguished connected component $\Stab^\dagger(X)$. Then
the orthogonal complent $^{\perp}\mathcal{E} = \mathcal{E}^\perp \subset \Db(X)$ of $\mathcal{E}$ contains non-trivial
objects.
\end{prop}

Note that by \cite[Proposition 14.2]{Bridgeland:K3}, this applies to every vector bundle that is Gieseker-stable for
some polarization. Moreover, it applies to all spherical objects for generic K3s:

\begin{cor} \label{cor:nonemptyspherical}
Assume that the K3 surface $X$ has Picard rank one. Then every spherical object $\mathcal{E} \in \Db(X)$ has
non-trivial orthogonal complement.
\end{cor}
\begin{proof}
By \cite[Corollary 6.9]{K3Pic1}, there exists a stability condition $\sigma \in \Stab^\dagger(X)$ for
which $\mathcal{E}$ is $\sigma$-stable.
\end{proof}

\subsection{Proof}
The Picard rank of $X$ is denoted by $\rho$.
Recall from Section \ref{sect:prelims} that we write $K_{\mathrm{num}}(X) \cong \widetilde{H}^{1,1}(X,\mathbb{Z})$
for the numerical $K$-group of $X$ , which is a lattice of signature $(2, \rho)$; for $\mathcal{E} \in D^b(X)$, we
write $v(\E)$ for its Mukai vector. 
The central charge $Z$ of a stability condition $\sigma = (Z, \mathcal{P}) \in
\Stab^\dagger(X)$ is a group homomorphism $Z \colon K_{\mathrm{num}}(X) \to \CC$; we write
$\mathrm{Ker} Z \subset K_{\mathrm{num}}(X) \otimes \RR$ for its kernel. 
By \cite[Theorem 1.1]{Bridgeland:K3}, $\mathrm{Ker} Z$ is negative definite with respect to the
Mukai pairing.  We write $s := v(\mathcal{E})$ for the Mukai vector of a given spherical
object $\mathcal{E}$.

\begin{lem} \label{lem:densitynonsquares}
Let $s^\perp_+ \subset K_{\mathrm{num}}(X) \otimes \mathbb{R}$ be the subset of classes $v$ with $(v, s) = 0$ and 
$v^2 > 0$. Then up to rescaling, 
the set of integral classes $v \in s^\perp_+$ such that $2 v^2$ is not a square
is dense in $s^\perp_+$. 
\end{lem}
Such statements are well-known. 
Note that $2v^2$ is a square iff the rank two lattice $\langle v, s\rangle$ contains
a vector of square zero.
\begin{proof}
Consider an integral class $v \in s^\perp_+$, and 
assume that $2v^2$ is a square. Let $u \in s^\perp \cap v^\perp$ be any integral class orthogonal to
both $s$ and $v$.
Then $2 (Nv + u)^2 = N^2 (2v^2) + 2 u^2$ is not a square for $N \gg 0$, and
$\frac 1N (Nv + u) \to v$ for $N \to +\infty$.
\end{proof}

\begin{proof}[Proof of Proposition \ref{prop:nonemptyspherical}]
Recall that stability of $\mathcal{E}$ is an open property, i.e.~it is preserved by small deformations of $\sigma$ \cite[Proposition
9.3]{Bridgeland:K3}. 

Consider the subspace $K  = \langle \mathrm{Ker} Z, s\rangle \subset K_{\mathrm{num}}(X)$ spanned by $s$ and the kernel
of $Z$; equivalently, $v \in K$ if and only if $Z(v)$ is proportional to $Z(s)$.
Since $\mathrm{Ker} Z$ has signature $(0, \rho)$ and
$K_{\mathrm{num}}(X)$ has signature $(2, \rho)$, the space $K$ must have signature $(1, \rho)$, and $s^\perp \cap
K$ has signature $(1, \rho-1)$. Thus, there is a class $v \in s^\perp \otimes \mathbb{R}$ with $v^2 > 0$
and $Z(v)$ proportional to $Z(s)$.

By Lemma \ref{lem:densitynonsquares}, we can deform $\sigma$ slightly and assume that
$v$ is integral and primitive, and such that $2v^2$ is not a square; by opennnes of stability,
$\mathcal{E}$ is still $\sigma$-stable. Replacing $v$ by $-v$ if necessary, we may also assume
that $Z(v)$ and $Z(s)$ lies on the same ray.

The claim now comes from the analysis of the wall-crossing of $\sigma$-stable objects of class $v$ 
in \cite{BM:walls}. To this end, let $\sigma_+ =
(Z_+, \mathcal{P}_+)$ be a stability condition nearby $\sigma$, such that the phase of $Z_+(v)$ is smaller
than the phase of $Z_+(s)$. Let  $M_{\sigma_+}(v)$ be the coarse moduli space of $\sigma_+$-stable
objects of Mukai vector $v$
The stability condition $\sigma$ lies on a wall $\mathcal{W}$ for $M_{\sigma_+}(v)$,
corresponding to the lattice $\mathcal{H}_\mathcal{W} = \langle s, v \rangle$ in the sense of \cite[Proposition
5.1]{BM:walls}: this means that for a generic stability condition on the wall $\mathcal{W}$, the integral classes in 
$\mathcal{H}_\mathcal{W}$ are the only integral classes sent to a complex number proportional to $Z(v)$.

By construction, $\mathcal{H}_\mathcal{W}$ does not contain an \emph{isotropic} class $w$ with $w^2 = 0$. Therefore,
we can apply \cite[Proposition 7.1]{BM:walls}; note that $s$ is \emph{effective} in the terminology
used [ibid.], as $Z(s)$ and $Z(v)$ lie on the same ray in the complex plane, see \cite[Proposition
5.5]{BM:walls}. 
Now \cite[Proposition 7.1]{BM:walls}
says that the moduli space $M_{\sigma_+}(v)$
contains divisor characterized by the property $\mathrm{Hom}(\ , \mathcal{E}) \neq 0$. (This is the divisor
contracted by the wall-crossing associated to the wall $\mathcal{W}$.) In particular, any object $F$ in the complement
of this divisor satisfies $\Hom(F, \mathcal{E}) = 0$. It also satisfies $\Hom(\mathcal{E}, F) = 0 = \Ext^2(F, \mathcal{E})$
by stability. Since $\mathcal{E}$ and $F$ are in the same heart of a K3 category, they also satisfy
$\Ext^{<0}(F, \mathcal{E}) = 0 = \Ext^{>2}(F, \mathcal{E}) = \Ext^{<0}(\mathcal{E}, F)$. Finally, since $\mathcal{H}^i(F, \mathcal{E}) = (v,
s) = 0$, we must also have $\Ext^1(F, \mathcal{E}) = 0$, and therefore $F \in ^\perp \mathcal{E}$.
 \end{proof}

\bibliography{all}                      
\bibliographystyle{halpha}     

 \end{document}